\documentclass{article}

\usepackage{amsfonts}
\usepackage{cite}
\usepackage{setspace}
\usepackage{epsfig}

\usepackage{graphicx} 
\usepackage{booktabs} 
\usepackage{amsmath}
\usepackage{color}
\usepackage{float}
\usepackage[font=small,labelfont=bf]{caption}
\usepackage{amssymb}
\usepackage{amsthm}

\def\vertex(#1){\put(#1){\circle*{2}}}
\def\vertexo(#1){\put(#1){\circle{2}}}
\def\vert(#1){\put(#1){\circle*{1.5}}}
\def\verto(#1){\put(#1){\circle{1.5}}}
\def\lab(#1)#2{\put(#1){\makebox(0,0)[c]{#2}}}

\newtheorem{theorem}{Theorem}
\newtheorem{lemma}{Lemma}
\newtheorem{conjecture}{Conjecture}
\newtheorem{corollary}{Corollary}

\begin{document}

\title{Global Cycle Properties in Locally Isometric Graphs}
\author{Adam Borchert, Skylar Nicol, Ortrud R. Oellermann\footnote{Supported by an NSERC grant CANADA}\\
 \small{Department of Mathematics and Statistics}\\
 \small{University of Winnipeg, Winnipeg MB, CANADA}\\
 \small{adamdborchert@gmail.com; skylarnicol93@gmail.com; o.oellermann@uwinnipeg.ca}}
\date{}

\maketitle

\begin{abstract}
Let $\cal P$ be a graph property. A graph $G$ is said to be {\em locally} $\cal P$ if the subgraph induced by the open neighbourhood of every vertex in $G$ has property $\cal P$.  Ryj\'{a}\v{c}ek's well-known conjecture that every connected, locally connected graph is weakly pancyclic motivated us to  consider the global cycle structure of locally $\cal P$ graphs, where $\cal P$ is the property of having diameter at most $k$ for some fixed $k \ge 1$. For $k=2$ these graphs are called \emph{locally isometric graphs}. For $\Delta \leq 5$, we obtain a complete structural characterization of locally isometric graphs that are not fully cycle extendable. For $\Delta = 6$, it is shown that locally isometric graphs that are not fully cycle extendable contain a pair of true twins of degree $6$. Infinite classes of locally isometric graphs with $\Delta = 6$ that are not fully cycle extendable are described and observations are made that suggest that a complete characterization of these graphs is unlikely. It is shown that  Ryj\'{a}\v{c}ek's conjecture holds for all locally isometric graphs with $\Delta \leq 6$. The Hamilton cycle problem for locally isometric graphs with maximum degree at most 8 is shown to be NP-complete.

\medskip


\end{abstract}

\section{Introduction}

Recent advances in graph theory have been influenced significantly by the rapid growth of the internet and corresponding large communication networks. Of particular interest are social networks such as Facebook where it is not uncommon for the friends, i.e., neighbours, of any given person to be themselves `closely' connected within this neighbourhood.  In this paper we study global cycle structures of graphs for which the neighbourhoods of all vertices induce graphs in which every two nodes are not `too far apart'. To make these notions more precise we begin by introducing some terminology and background on closely related literature.

Let $G$ be a graph. The order of $G$ is denoted by $n(G)$ and the minimum and maximum degree of  $G$ is denoted by $\delta(G)$ and $\Delta(G)$, respectively. If $G$ is clear from context we use $n$, $\delta$ and $\Delta$, to denote these respective quantities. We say that $G$ is \emph{Hamiltonian} if $G$ has a cycle of length $n$. If, in addition, $G$ has a cycle of every length from $3$  to $n$, then $G$ is \emph{pancyclic}. An even stronger notion than pancyclicity is that of `full cycle extendability', introduced by Hendry \cite{H1}. A cycle $C$ in a graph $G$ is {\em extendable} if there exists a cycle $C'$ in $G$ that contains all the vertices of $C$ and one additional vertex.  A graph $G$ is {\em cycle extendable} if every nonhamiltonian cycle of $G$ is extendable. If, in addition, every vertex of $G$ lies on a 3-cycle, then $G$ is \emph{fully cycle extendable.}

A graph $G$ that is not necessarily Hamiltonian but has cycles of every possible length from the shortest cycle length $g(G)$ (called the {\em girth} of $G$) to the longest cycle length $c(G)$ (called the {\em circumference} of $G$) is said to be {\em weakly pancyclic}.

By a \emph{local property} of a graph we mean a property that is shared by the subgraphs induced by all the open neighbourhoods of the vertices.
We use $N(v)$ to denote the open neighbourhood of a vertex $v\in V(G)$. If $X\subseteq V(G)$, the subgraph induced by $X$ is denoted by $\langle X\rangle$. For a given graph property $\cal P$, we call a graph $G$ \emph{locally $\cal P$} if $\langle N(v) \rangle$ has property $\cal P$ for every $v\in V(G)$. Skupie\'{n} \cite{S2} defined a graph $G$ to be \emph{locally Hamiltonian} if $\langle N(v) \rangle$ is Hamiltonian for every $v\in V(G)$.  Locally Hamiltonian graphs were subsequently studied, for example, in \cite{Pa,PS,S}. A graph is {\em traceable } if it has a Hamiltonian path. Pareek and Skupie\'{n} \cite{PS} considered locally traceable graphs and  Chartrand and Pippert \cite{CP} introduced locally connected graphs. The latter have since been studied extensively - see for example \cite{CGP,CP,C,GOPS,H1,H2,OS}.

A classic example of a local property that guarantees Hamiltonicity is Dirac's minimum degree condition `$\delta(G)  \geq n(G)/2$' (see \cite{D}), which may be written as `$|N(v)|\geq n(G)/2$ for every vertex $v$ in $G$'. Bondy \cite{B} showed that Dirac's minimum degree condition actually guarantees more than just the existence of a Hamilton cycle by showing that these graphs are either pancyclic or isomorphic to the complete balanced bipartite graph $K_{n/2,n/2}$.

Another local property that is often studied in connection with Hamiltonicity is the property of being claw-free, i.e., not having the claw $K_{1,3}$ as induced subgraph. Note that a graph $G$ is claw-free if and only if $\alpha (\langle N(v)\rangle)\leq 2$ for every $v\in V(G)$ (where $\alpha$ denotes the vertex independence number).

It is well known that the {\em Hamilton Cycle Problem} (i.e., the problem of deciding whether a graph has a Hamiltonian cycle) is NP-complete, even for claw-free graphs.  However, Oberly and Sumner \cite{OS} showed that, connected, locally connected claw-free graphs are Hamiltonian. Clark \cite{C} strengthened this result by showing that under the same hypothesis these graphs are in fact pancyclic and Hendry \cite{H1}  observed that Clark had in fact shown that these graphs are fully cycle extendable.
These results support Bondy's well-known `meta-conjecture' (see \cite{B2}) that almost any condition that guarantees that a graph has a Hamilton cycle usually guarantees much more about the cycle structure of the graph. If the claw-free condition is dropped, Hamiltonicity is no longer guaranteed. In fact, Pareek and Skupie{\'n} \cite{PS} observed that there exist infinitely many connected, locally Hamiltonian graphs that are not Hamiltonian. However, Clark's result led Ryj\'{a}\v{c}ek to suspect that every locally connected graph has a rich cycle structure, even if it is not Hamiltonian. He proposed the following conjecture (see \cite{WR}.)

 \begin{conjecture} \label{ryjacek}
 (Ryj\'{a}\v{c}ek) Every locally connected graph is weakly pancyclic.
 \end{conjecture}

 Ryj\'{a}\v{c}ek's conjecture seems to be very difficult to settle. The conjecture for $\Delta \le 5$ was settled in \cite{AFOW}. However, even for graphs with  $\Delta= 6$ this problems is largely unsolved. Weaker forms of this conjecture are considered in \cite{AFOW} where it was shown that locally Hamiltonian graphs with  $\Delta = 6$ are fully cycle extendable.

 The Hamilton Cycle Problem for graphs with small maximum degree remains difficult, even when additional structural properties are imposed on the graph. For example, the Hamilton Cycle Problem is NP-complete for bipartite planar graphs with $\Delta \leq 3$ (see \cite{ANS}), for $r$-regular graphs for any fixed $r$ (see \cite{P}) and even for planar cubic 3-connected
 claw-free graphs (see \cite{LCM}).

Some progress has been made for locally connected graphs with small maximum degree. The first result in this connection was obtained by Chartrand and Pippert \cite{CP}. They showed that if $G$ is a connected, locally connected graph of order at least $3$ with $\Delta (G) \le 4$, then $G$ is either Hamiltonian or isomorphic to the complete $3$-partite graph $K_{1,1,3}$.  Gordon, Orlovich, Potts and Strusevich \cite{GOPS} strengthened their result by showing that apart from $K_{1,1,3}$ all connected, locally connected graphs with maximum degree at most 4 are in fact fully cycle extendable. Since $K_{1,1,3}$ is weakly pancyclic,  Ryj\'{a}\v{c}ek's conjecture holds for locally connected graphs with maximum degree at most $4$.

Global cycle properties of connected, locally connected graphs with maximum degree $5$ were investigated in \cite{GOPS,H2,K}. Their combined results imply that every connected, locally connected graph with $\Delta =5$ and $\delta \geq 3$ is fully cycle extendable. This is not the case if $\Delta = 5$ and $\delta =2$. As mentioned above these graphs are weakly pancyclic, but  infinitely many are nonhamiltonian as was shown in \cite{AFOW}.  It was also shown in \cite{AFOW} that if `local connectedness' is replaced by `local traceability' in graphs with $\Delta =5$, then apart from three exceptional cases all these graphs are fully cycle extendable.
With a stronger local condition, namely `local Hamiltonicity', it was shown in \cite{AFOW} that an even richer cycle structure is guaranteed. More specifically it was shown  that every connected, locally Hamiltonian graph with $\Delta\leq 6$ is fully cycle extendable.

Let $G$ be a locally connected graph with maximum degree $\Delta$.  Then for every $v \in V(G)$, $diam(\langle N(v) \rangle) \le \Delta -1$. We define a graph $G$ to be \emph{locally} $k$-\emph{diameter bounded}  if $diam(\langle N(v) \rangle) \le k$ for all $v \in V(G)$.  If, for a given fixed value of $\Delta$, and every $k$, $1 \le k \le \Delta-1$, it is possible to determine (efficiently) when a locally $k$-diameter bounded graph is Hamiltonian, then the Hamilton cycle problem of locally connected graphs with maximum degree $\Delta$ can be solved (efficiently).  We observe that a locally traceable graph with maximum degree $\Delta$ is locally $(\Delta -1)$-diameter bounded and that a locally Hamiltonian graph with maximum degree $\Delta$ is $\lfloor \Delta/2 \rfloor$-diameter bounded. These observations motivate the study of the cycle structure of locally $k$-diamater bounded graphs with `small' maximum degree.

For $k=2$, the locally $k$-diameter bounded graphs have the property that $\langle N(v) \rangle$  is an \emph{isometric}  (i.e. a distance preserving) subgraph of $G$ for all $v \in V(G)$. We thus refer to locally $2$-diameter bounded graphs as \emph{locally isometric} graphs.  In this paper we show that all locally isometric graphs with $\Delta \le 6$ are weakly pancyclic. For $\Delta \le 5$ we give complete structural characterizations of those locally isometric graphs that are not fully cycle extendable.  We show that every locally isometric graph with $\Delta = 6$ that does not contain a pair of true twins of degree $6$ is fully cycle extendable. We conclude by showing that the Hamilton cycle problem for (i) locally $3$-diameter bounded graphs having maximum degree $7$  and (ii) locally isometric graphs with $\Delta =8$ is NP-complete.

\section{Definitions and Preliminary Results}

In this section we introduce definitions and results used in the sequel.  Graph theory terminology not given here can be found in \cite{bm}. If $G$ and $H$ are vertex disjoint graphs, then the {\em join} of $G$ and $H$, denoted by $G+H$, is the graph obtained from the union of $G$ and $H$ by adding all possible edges between the vertices of $G$ and the vertices of $H$.  Let $G$ be a graph. Suppose $x$ and $y$ are vertices of a graph $G$ and that $V_1$ and $V_2$ are subsets of the vertex set of $G$. We use $ x \sim y$ to mean that $x$ is adjacent with $y$ and $x \nsim y$ to mean that $x$ is not adjacent with $y$. Also $x \sim V_1$ means that $x$ is adjacent with every vertex of $V_1$ and $x \nsim V_1$ means that $x$ is not adjacent with any vertex of $V_1$ while $V_1 \sim V_2$ (and $V_1 \nsim V_2$) means that every vertex of $V_1$ is adjacent with every vertex of $V_2$  (no vertex of $V_1$ is adjacent with any vertex of $V_2$, respectively). Two vertices $u$ and $v$ of a graph $G$ are \emph{false twins} if  $N(u)=N(v)$,  and \emph{true twins} if $N[u]=N[v]$. We say a graph $H$ is obtained from a graph $G$ by {\em false twinning} (respectively, {\em true twinning}) a vertex $v$ of $G$, if $H$ is obtained from $G$ by adding a new vertex $v'$ to $G$ and joining it to each neighbour of $v$ in $G$ (respectively, joining it to $v$ and each neighbour of $v$ in $G$).  A vertex $v$ of a graph $G$ is a \textit{universal vertex} of $G$ if $v \sim (V(G) - \{v\})$.

We now introduce families of locally isometric graphs that are not fully cycle extendable. These graphs will be referred to in Sections 3 and 4.   The \emph{strong product} of two graphs $G$ and $H$, denoted by $G \boxtimes H$, has as its vertex set the Cartesian product of the vertex sets of $G$ and $H$ and two vertices $(u,v)$ and $(x,y)$ are adjacent if and only if $u=x$ and $vy \in E(H)$, or $v=y$ and $ux \in E(G)$, or $ux \in E(G)$ and $vy \in E(H)$. For $k\ge 2$, the \emph{highrise} graph of order $2k+1$, denoted by $H_{2k+1}$, is obtained from $P_2 \boxtimes P_k$ by adding a new vertex $u$ and joining it to a pair of adjacent vertices having degree $3$ and a highrise graph of order $2k+2$, denoted by $H_{2k+2}$, is obtained from $P_2 \boxtimes P_k$ by adding two new vertices and joining one of these to one pair of adjacent vertices of degree $3$ and the other to the other pair of adjacent vertices of degree $3$. It is readily seen that the highrise graphs are locally isometric and fully cycle extendable. Let $\Delta \ge 5$, $r \ge \Delta -4$ and $m \ge  r+5$ be integers.  The family of $r$-{\emph shuttered} highrise graphs of order $m$ and maximum degree $\Delta$, denoted by ${\mathcal{H}}(m,r,\Delta)$, are those graphs obtained by adding $r$ new vertices to the highrise graph $H_{m-r}$ and joining each new vertex to a pair of true twins of degree at least $4$ in $H_{m-r}$ in such a way that the resulting graph has maximum degree $\Delta$. In order to be able to construct such a graph we observe that if $m-r$ is even, then $H_{m-r}$ has $\frac{m-r-2}{2}$ pairs of true twins all of which have degree 5, except two pairs of true twins which have degree 4. So $(\frac{m-r-2}{2})(\Delta -5) \ge r-2$. Similarly we can argue that if $m-r$ is odd, then $(\frac{m-r-3}{2})(\Delta-5) \ge r-1$. Thus if $\Delta = 5$, then $r \le 2$ if $m-r$ is even and $r \le 1$ if $m-r$ is odd.  It now follows that there is, up to isomorphism, exactly one graph in ${\mathcal{H}}(m,1,5)$  and exactly one graph in ${\mathcal{H}}(m,2,5)$. The graph in ${\mathcal{H}}(m,1,5)$ is called a \emph{singly shuttered highrise} of order $m \ge 6$, and is denoted by $S_m$, and the graph in ${\mathcal{H}}(m,2,5)$ is called a \emph{doubly shuttered highrise} of order $m \ge 8$, and is denoted by $D_m$.  For convenience we will also refer to $S_5=K_2 + \overline{K}_3$ as a singly shuttered high rise (of order $5$) with maximum degree $4$ and $ D_6=K_2 + \overline{K}_4$  as a doubly shuttered highrise (of order $6$) with maximum degree 5. Note that a singly shuttered highrise may have even or odd order while a doubly shuttered highrise has even order. Figure \ref{shutteredhighrise} shows $H_7$, $S_8$ and $D_{10}$. Note that $S_9$ can be obtained from $D_{10}$ by deleting a vertex of degree $2$.

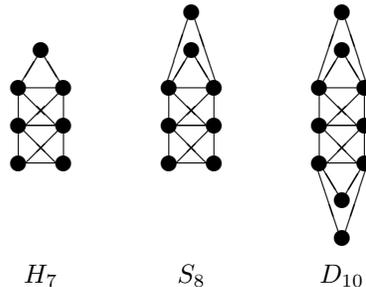
\begin{figure}[htb]
\begin{center}
\begin{picture}(-5,10)

\multiput(0,10)(20,0){2}{\circle*2}
\multiput(-20,5)(20,0){3}{\circle*2}

\multiput(-23,0)(6,0){2}{\circle*2}
\multiput(-3,0)(6,0){2}{\circle*2}
\multiput(17,0)(6,0){2}{\circle*2}

\multiput(-23,-5)(6,0){2}{\circle*2}
\multiput(-3,-5)(6,0){2}{\circle*2}
\multiput(17,-5)(6,0){2}{\circle*2}

\multiput(-23,-10)(6,0){2}{\circle*2}
\multiput(-3,-10)(6,0){2}{\circle*2}
\multiput(17,-10)(6,0){2}{\circle*2}

\put(20,-15){\circle*2}
\put(20,-20){\circle*2}

\put(-23,0){\line(3,5){3}}
\put(-17,0){\line(-3,5){3}}

\put(-23,0){\line(0,-1){5}}
\put(-17,0){\line(0,-1){5}}
\put(-23,0){\line(1,0){5}}
\put(-23,-5){\line(1,0){5}}
\put(-23,-6){\line(1,1){6}}
\put(-17,-6){\line(-1,1){6}}

\put(-23,-5){\line(0,-1){5}}
\put(-17,-5){\line(0,-1){5}}
\put(-23,-5){\line(1,0){5}}
\put(-23,-10){\line(1,0){5}}
\put(-23,-11){\line(1,1){6}}
\put(-17,-11){\line(-1,1){6}}

\put(-3,0){\line(3,5){3}}
\put(3,0){\line(-3,5){3}}

\put(-3,0){\line(0,-1){5}}
\put(3,0){\line(0,-1){5}}
\put(-3,0){\line(1,0){5}}
\put(-3,-5){\line(1,0){5}}
\put(-3,-6){\line(1,1){6}}
\put(3,-6){\line(-1,1){6}}

\put(-3,-5){\line(0,-1){5}}
\put(3,-5){\line(0,-1){5}}
\put(-3,-5){\line(1,0){5}}
\put(-3,-10){\line(1,0){5}}
\put(-3,-11){\line(1,1){6}}
\put(3,-11){\line(-1,1){6}}

\put(17,0){\line(3,5){3}}
\put(23,0){\line(-3,5){3}}

\put(17,0){\line(0,-1){5}}
\put(23,0){\line(0,-1){5}}
\put(17,0){\line(1,0){5}}
\put(17,-5){\line(1,0){5}}
\put(17,-6){\line(1,1){6}}
\put(23,-6){\line(-1,1){6}}

\put(17,-5){\line(0,-1){5}}
\put(23,-5){\line(0,-1){5}}
\put(17,-5){\line(1,0){5}}
\put(17,-10){\line(1,0){5}}
\put(17,-11){\line(1,1){6}}
\put(23,-11){\line(-1,1){6}}

\put(-3.5,0){\line(1,3){3}}
\put(3.5,0){\line(-1,3){3}}

\put(16.5,0){\line(1,3){3}}
\put(23.5,0){\line(-1,3){3}}

\put(16.5,-10){\line(1,-3){3}}
\put(23.5,-10){\line(-1,-3){3}}

\put(17,-10){\line(3,-5){3}}
\put(23,-10){\line(-3,-5){3}}


\lab(-20,-25){$H_7$}
\lab(0,-25){$S_8$}
\lab(20,-25){$D_{10}$}


\end{picture}
\end{center}
\vskip 2 cm \caption{Graphs of a highrise, singly shuttered highrise and doubly shuttered highrise}
\label{shutteredhighrise}
\end{figure}

Let $C=v_0v_1v_2\ldots v_{t-1} v_0$ be a $t$-cycle in a graph $G$. If $i \ne j$ and $\{i,j\}\subseteq \{0,1,\ldots, t-1\}$, then $v_i\overrightarrow{C}v_j$ and $v_i\overleftarrow{C}v_j$ denote, respectively, the paths $v_iv_{i+1}\ldots v_j$ and $v_iv_{i-1}\ldots v_j$ (subscripts expressed modulo $t$).  Let $C=v_0 v_1, \ldots v_{t-1} v_1$ be a non-extendable cycle in a graph $G$. With reference to a given non-extendable cycle $C$, a vertex of $G$ will be called a \emph{cycle vertex} if it is on $C$, and an \emph{off-cycle }vertex if it is in $V(G)-V(C)$. A cycle vertex that is adjacent to an off-cycle vertex will be called an \emph{attachment vertex}.
The following lemma on non-extendable cycles was established in \cite{AFOW}. For completeness, we include its proof.

\begin{lemma} \label{lemma:nonextendable} Let $C=v_0v_1\ldots v_{t-1} v_0$ be a non-extendable cycle of length $t$ in a graph $G$. Suppose $v_i$ and $v_j$ are two distinct attachment vertices of $C$ that have a common off-cycle neighbour $x$. Then the following hold. (All subscripts are expressed modulo $t$.)

\begin{itemize}
\item[\emph{1}.] $j\neq i+1$ and $j \ne i-1$.
\item[\emph{2}.] $v_{i+1} \nsim v_{j+1}$ and $v_{i-1} \nsim v_{j-1}$.

\item[\emph{3}.] If $v_{i-1} \sim v_{i+1}$, then $v_{j-1} \nsim v_i$ and $v_{j+1} \nsim v_i$.
\item[\emph{4}.] If $j=i+2$ then $v_{i+1}$ does not have two adjacent neighbours $v_k,v_{k+1}$ on the path  $v_{i+2}\overrightarrow{C}v_i$.
\end {itemize}
\end{lemma}

\begin{proof} We prove each item by presenting an extension of $C$ that would result if the given statement is assumed to be false. For (2) and (3) we only establish the first non-adjacency. The second non-adjacency can be proven similarly.
\begin{enumerate}
\item $v_ixv_{i+1}\overrightarrow{C}v_i$.
\item $v_{i+1}v_{j+1}\overrightarrow{C}v_ixv_j\overleftarrow{C}v_{i+1}$.
\item $v_{j-1}v_ix v_j\overrightarrow{C}v_{i-1}v_{i+1}\overrightarrow{C}v_{j-1}$.
\item $v_kv_{i+1}v_{k+1}\overrightarrow{C}v_ixv_{i+2}\overrightarrow{C}v_k$.
\end{enumerate}
\end{proof}

\begin{lemma}
\label{lem5}
Let $C = v_0v_1v_2...v_{t-1}v_0$ be a non-extendable cycle of length t in a locally isometric graph G with $\Delta = 6$. Suppose $v_0$ has exactly one off-cycle neighbour $x$ and that $N(v_0) = \{v_1, v_i, v_j, v_k, v_{t-1}, x\}$ where $ 1 < i < j < k < t-1$.Then the following hold:
\begin{enumerate}
\item If $v_k \sim \{x, v_1, v_{t-1}\}$, then $k+1 = t-1$. Similarly, if $v_1 \sim \{x,v_1,v_{t-1}\}$, then $i=2$.
\item If $v_j \sim \{x, v_1, v_{t-1}\}$, then $v_0$ and $v_j$ are true twins of degree 6.
\end{enumerate}
\end{lemma}
\begin{proof}
\begin{enumerate}
\item We prove the first of these two statements. The second can be proved similarly. Let $v_k \sim \{x, v_1, v_{t-1}\}$ and suppose $k+1 \neq t-1$. By Lemma \ref{lemma:nonextendable}(1), $x \nsim \{v_{t-1}, v_{k+1}\}$, and by Lemmas  \ref{lemma:nonextendable}(2) and \ref{lemma:nonextendable}(3), and by $\Delta = 6$, $v_{k+1} \nsim \{v_1, v_{k-1}, v_0\}$. But then $d_{\langle N(v_k)\rangle}(v_{k+1}, x) >2$, contradicting the local isometry of the graph. Thus $k+1 = t-1$.
\item Let $v_j \sim \{x, v_1, v_{t-1}\}$. By Lemma \ref{lemma:nonextendable}(1), $x \nsim \{v_1, v_{j-1}, v_{j+1}, v_{t-1}\}$, so since $diam\langle N(v_j)\rangle \leq 2$, $v_0 \sim \{v_{j-1}, v_{j+1}\}$. Hence by $\Delta = 6$, we must have $j-1 = i$ and $j+1 = k$. But then $v_j$ is a true twin of $v_0$ and $deg(v_0) = deg(v_j) = 6$.
\end{enumerate}
\end{proof}

\begin{lemma}
\label{lem8}
Let $G$ be a locally isometric graph of order $n \ge 3$ with no true twins of degree $\Delta$. Suppose $x$ and $v$ are vertices of $G$ such that $x \in N(v)$ and $deg(v) = \Delta$. Then there exist distinct vertices $y, z \in N(v) - \{x\}$ such that $x \sim \{y,z\}$.
\end{lemma}
\begin{proof}
Suppose $x$ is adjacent to at most one other vertex in $N(v)$, say $y$. Since $diam\langle N(v)\rangle \leq 2$, $y$ must be a universal vertex in $\langle N(v)\rangle$. So $v$ and $y$ are true twins of degree $\Delta$, a contradiction.
\end{proof}

\begin{lemma}
\label{true_twin}
Let G be a locally isometric graph and let $x, v \in V(G)$ be such that $deg(v)=k$ and $x \in N(v)$. If there exists a set $S\subseteq N(v)-\{x\}$ such that $x\nsim S$ and $|S|=k-2$, then there is a vertex $y\in N(v)$ which is universal in $\langle N(v)\rangle$. If $k=\Delta$,  $y$ and $v$ are true twins.
\end{lemma}
\begin{proof} Let $x$ and $v$ be vertices of $G$ that satisfy the hypothesis of the lemma. Since $G$ is locally isometric, the vertex of $N(v) -(S \cup \{x\})$, call it $y$, must be adjacent with $x$ and every vertex of $S$. Hence this vertex is universal in $\langle N(v) \rangle$. If $k =\Delta$, then $y$ and $v$ must be true twins since $y$ cannot have any neighbours in $G$ other than those in $N[v]$.
\end{proof}

Let $S$ be a set of vertices in a graph $G$. Then $N(S)=\cup \{N(x)~|~ x \in S\}$.

\begin{lemma} \label{full_neighbour}
Let $G$ be a locally isometric graph and $S$ a set of vertices of $G$. If $w \sim S$ and $w \nsim (N(S)-(S \cup\{w\}))$, then $N(w)=S$.
\end{lemma}
\begin{proof}
Suppose $w \sim y$ where $y \not\in S$. By the hypothesis $y \not\in (N(S)-(S \cup \{w\}))$. Since $(N(S)-(S \cup\{w\})) \cap N(w)= \emptyset$, it follows that if $x \in S$, then there is no path from $y$ to $x$ in $\langle N(v) \rangle$. This contradicts the fact that $\langle N(v) \rangle$ is an isometric and hence connected subgraph of $G$.
\end{proof}

\section{Cycles in Locally Isometric Graphs with $\Delta \leq 5$}

In this section we study the cycle structure of locally isometric graphs having maximum degree at most $5$.  It is readily seen that a graph is locally isometric if and only if each of its components is locally isometric. Moreover, if $G$ is a locally isometric graph with circumference $c(G)$, then $G$ is weakly pancyclic if and only if some component containing a cycle of length $c(G)$ is weakly pancyclic. We show that all connected locally isometric graphs with $\Delta \le 5$ that are not fully cycle extendable are weakly pancyclic. Observe that every vertex of a locally connected graph is contained in a 3-cycle. Thus to establish Ryj\'{a}\v{c}ek's conjecture, it suffices to show that every connected locally connected graph $G$ contains a cycle of length $k$ for every $k$, $3 \leq k \leq c(G)$.

Let $G$ be a connected locally isometric graph of order $n \ge \Delta +1$. If $\Delta = 1$ or $2$, $G$ is isomorphic to $K_2$ or $K_3$, respectively and if $\Delta = 3$, then $G$ is isomorphic to $K_4$ or $K_4-e$ where $e$ is any edge of $K_4$.

Let $G$ be a locally isometric graph with $\Delta = 4$ that is not fully cycle extendable. Let $C=v_0v_1v_2 \ldots v_{t-1} v_0$ be a non-extendable cycle in $G$. Since $G$ is connected, some vertex of $C$ is an attachment vertex. We may assume that $v_0$ is an attachment vertex with off-cycle neighbour $x$. By Lemma \ref{lemma:nonextendable}(1) $x \nsim \{v_{t-1},v_1\}$. Since $G$ is locally isometric it follows that $v_0$ has a neighbour $u \in N(v)-\{x,v_1,v_{t-1}\}$ such that $u \sim \{x, v_1, v_{t-1}\}$. Since $u \sim \{v_0, v_1\}$, it follows from Lemma \ref{lemma:nonextendable}(1) that $u$ is not an off-cycle neighbour of $v_0$. So $v_0$ has a cycle neighbour $v_i$, $1 < i < t-1$, such that $v_i \sim \{x, v_1, v_{t-1}\}$. Since $\Delta = 4$, it follows that $i-1=1$ and $i+1=t-1$. Hence $t=4$. By Lemma \ref{lemma:nonextendable}(2) $v_1\nsim v_3$. Since $\Delta =4$, it follows from Lemma \ref{full_neighbour} that none of $v_1,v_3$ and $x$ is adjacent with a vertex not in $V(C) \cup\{x\}$. Hence if $G$ is a locally isometric graph with $\Delta = 4$ that is not fully cycle extendable, then $G \cong K_2 + \overline{K}_3$. Thus every connected locally isometric graph with $\Delta \le 4$ that is not fully cycle extendable is weakly pancyclic.

\begin{theorem}
\label{thm3}
Let $G$ be a connected locally isometric graph of order $n \ge 6$ and maximum degree $\Delta = 5$. Then $G$ is either fully cycle extendable or  $G$ is  a singly  or a doubly shuttered highrise.
\end{theorem}
\begin{proof}
Let $C=v_0v_1v_2 \ldots v_{t-1} v_0$ be a non-extendable cycle in $G$. We begin by showing that there is an attachment vertex of degree 5. Suppose that every attachment vertex has degree at most $4$. Then each such vertex has degree exactly $4$. We may assume that $v_0$ is adjacent with an off-cycle vertex $x$. As above we see that there is a vertex $v_i$ on $C$ for $1 < i < t-1$ such that $v_i \sim \{x,v_0,v_1,v_{t-1}\}$. From our assumption $t=4$. Moreover, we now see that $G \cong K_2 + \overline{K}_3$. This contradicts the fact that $\Delta =5$.

So we may assume that $v_0$ is an attachment vertex of degree $5$.  Suppose first that $v_0$ has two off-cycle neighbours $x$ and $y$. Since $\langle N(v_0) \rangle$ is an isometric subgraph of $G$ and $C$ is not extendable, we have, by Lemma \ref{lemma:nonextendable}(1), $\{x,y\} \nsim \{v_1, v_{t-1}\}$. So there is a vertex $v_i$ on $C$ such that $v_i \sim \{x,y,v_0,v_1,v_{t-1}\}$. Since $\Delta = 5$ it follows that $t=4$ and by Lemma \ref{lemma:nonextendable}(2) $v_1 \nsim v_3$.  By Lemma \ref{full_neighbour}, the only neighbours of $v_1$ and $v_3$ in $G$ are $v_0$ and $v_2$. If  $x \nsim y$, then $\{x,y,v_1, v_3\}$ is an independent set. By Lemma \ref{full_neighbour}, $ G \cong K_2+ \overline{K}_4$. Thus, in this case, $G$ is a doubly shuttered high rise of maximum degree $5$. If $x \sim y$, then either $G \cong K_2+(K_2 \cup \overline{K}_2)$, in which case $G$ is a singly shuttered highrise of maximum degree $5$, or one of $x$ and $y$ is adjacent with a vertex not in $T_0=V(C) \cup\{x,y\}$. Since $\Delta =5$ and as $G$ is locally isometric we see that every neighbour of $x=x_0$ not in $T_0$ must also be a neighbour of $y=y_0$. So $N[x_0] = N[y_0]$. If $N(x_0)-T_0$ contains exactly one vertex, $G$ is isomorphic to $S_7$ and if $N(x_0) - T_0$ contains two vertices $x_1$ and $y_1$, say, then $G$ is either isomorphic to $S_8$ or $D_8$ or at least one of $x_1$ and $y_1$ is adjacent with a vertex not in $T_1=T_0 \cup \{x_1,y_1\}$. In the latter case we see, as for $x_0$ and $y_0$, that $N[x_1]=N[y_1]$.  Suppose $G$ has order at least 8 and that $T_i =T_{i-1} \cup\{x_i, y_i\}$ has been determined for some $i \ge 1$ where $N[x_i] = N[y_i]$. If $N(x_i) -T_i$ is empty, then $G \cong S _{2i+6}$ and if $N(x_i) -T_i$ contains exactly one vertex, then $G \cong S_{2i+7}$. Otherwise $N(x_i) -T_i$ contains exactly two vertices $x_{i+1}$ and $y_{i+1}$ both of which must be adjacent to $x_i$ and $y_i$. If $x_{i+1} \nsim y_{i+1}$, then $G \cong D_{2i+8}$; otherwise, $N[x_{i+1}]=N[y_{i+1}]$. Since $G$ is finite, this process will terminate. Thus $G$ is isomorphic to either $S_n$ or $D_n$ for some $n \ge 6$. In all cases $G$ is weakly pancyclic.

We now assume that every cycle vertex has at most one off-cycle neighbour. Moreover, we assume that $deg(v_0)=5$ and that $x$ is the off-cycle neighbour of $v_0$. Let $\{x,v_1,v_i,v_j,v_{t-1}\}$ be the neighbours of $v_0$ where $1<i<j<t-1$. By Lemma \ref{lemma:nonextendable}(1) $x \nsim \{v_1,v_{t-1}\}$.

Suppose first that $x$ is adjacent with both $v_i$ and $v_j$. So, by Lemma \ref{lemma:nonextendable}(1), $j \ne i+1$.  Assume first that $v_i \sim v_j$. Hence $x,v_0,v_{i-1},v_{i+1}, v_j$ are five distinct neighbours of $v_i$. Since $\Delta =5$, these are precisely the neighbours of $v_i$. Similarly  $x,v_0,v_i,v_{j-1},v_{j+1}$ are precisely the neighbours of $v_j$. Since $v_1$ and $v_{t-1}$ must each be adjacent with at least one of $v_i$ or $v_j$, it follows that $i=2$ and $j=t-2$. By Lemma \ref{lemma:nonextendable}(2) $v_1 \nsim v_{t-1}$. Thus the distance between $v_1$ and $v_{t-1}$ in $\langle N(v_0) \rangle$ is $3$. This contradicts the fact that $G$ is locally isometric.

We now assume that $v_i \nsim v_j$.
Suppose $v_i \sim v_{t-1}$. Since $x \nsim \{v_{i-1},v_{i+1},v_{t-1}\}$ it follows by Lemma \ref{true_twin} that $v_{i+1} \sim v_0$ which is not possible. So $v_i \nsim v_{t-1}$. Similarly $v_j \nsim v_1$. Since $\langle N(v_0) \rangle$ is an  isometric  subgraph, it follows that $v_1 \sim v_i$ and $v_{t-1} \sim v_j$. If $i \ne 2$, then as before $v_{i+1} \sim v_0$ which is not possible. So $i=2$. Similarly $j=t-2$. By Lemma \ref{lemma:nonextendable}(2) $v_1 \nsim v_{t-1}$.  But then $\langle N(v_0) \rangle$ is not an isometric subgraph of $G$ since the distance between $v_1$ and $v_{t-1}$ is 4. Hence $x$ cannot be adjacent with both $v_i$ and $v_j$.

We may thus assume, without loss of generality, that $x$ is adjacent with $v_i$ and not with $v_j$. By Lemma \ref{true_twin}, $N[v_i]=N[v_0]$. Since $\Delta=5$,  it follows that $i=2$ and $j=i+1=3$ and, by Lemma \ref{lemma:nonextendable}(2), that $v_1 \nsim \{v_3, v_{t-1}\}$.  By Lemma \ref{full_neighbour}, $N(x) =\{v_0, v_2\}$. Hence, $deg (x) =2$. Similarly $deg(v_1) = 2$.

Suppose $k \ge 1$ is an integer such that $t \ge 2k+3$. We show by induction on $k \ge 1$, that the following statement holds: $P(k)$: either (i) $k+3 = t-k$ in which case $G \cong S_{2k+4}$ or (ii) $k+4 = t-k$, in which case $G \cong S_{2k+5}$ or $G \cong D_{2k+6}$, or (iii) $k+4 < t-k$ and $v_{k+3} \sim \{v_{k+2}, v_{t-k}\}$ and $v_{t-k-1} \sim \{v_{k+2}, v_{t-k}\}$ and $deg(v_{k+2}) =deg(v_{t-k})=5$.

Let $k=1$. If $k+3 =4 = t-1$, then necessarily $t=5$. Since $C$ is not extendable, neither $v_3$ nor $v_4$ is adjacent with an off-cycle neighbour. Thus $G \cong S_6$ and (i) holds. If $k+4=5=t-1$, then $t=6$. Since $\langle N(v_3) \rangle$ is isometric,  $v_3$ and $v_4$ have a common neighbour with $v_2$. Hence $v_3$ and $v_4$ are both adjacent with $v_5$.  If $v_3$ is adjacent with some off-cycle vertex $y$, then $y \sim v_5$ since $\langle N(v_3) \rangle$ is an isometric subgraph. In this case both $v_3$ and $v_5$ have degree 5 and $N[v_3]=N[v_5]=\{v_0,v_2,v_3,v_4,v_5, y\}$. Since $\Delta = 5$, these two vertices cannot be adjacent to any vertices not already specified. Also we can argue as before that in this situation $deg(v_4) =2$. So in this case $G \cong D_{8}$. Suppose now that $v_3$, and by symmetry $v_5$, have no off-cycle neighbours. Then $deg(v_3) =deg(v_5) =4$. In this case $deg(v_4) =2$. So $G \cong S_{7}$. So if $k+4=5=t-1$, then (ii) holds.

Suppose now that $t>6$. Since $\langle N(v_3) \rangle$ is an isometric subgraph and $v_4 \nsim v_0$, it follows that $v_3$ and $v_4$ share a common neighbour with $v_0$. So $v_{t-1} \sim \{v_3, v_4\}$. Similarly, $v_{3} \sim v_{t-2}$. So $v_4 \sim \{v_3,v_{t-1} \}$ and $v_{t-2} \sim \{ v_{t-1}, v_{3} \}$. Also $deg(v_3) =deg(v_{t-1}) = 5$. So (iii) holds in this case.

Assume next that $m >1$ and that $t \ge 2m+3$ and that $P(m-1)$ holds. So $t-m \ge m+3$. Hence (i) and (ii) of $P(m-1)$ do not hold. Thus condition (iii) of $P(m-1)$ is true. So $v_{m+2} \sim \{v_{m+1}, v_{t-m+1}\}$  and $v_{t-m} \sim \{v_{m+1}, v_{t-m+1}\}$. Moreover, $deg(v_{m+1})=5$ and $deg(v_{t-m+1}) = 5$. If $m+3=t-m$, then it follows, since $C$ is not extendable, that $G \cong S_{2m+4}$. So (i) holds in this case. Suppose next that $m +4 =t-m$. Since $\langle N(v_{m+2}) \rangle$ is an isometric subgraph $v_{m+2} \sim v_{t-m}$. Since $C$ is not extendable, $v_{m+3}$ has degree 2. If $v_{m+2}$ (and thus $v_{t-m}$) is not adjacent with an off-cycle neighbour, then $G \cong S_{2m+5}$. If $v_3$ has an off-cycle neighbour $y$, then it is seen as before that $y \sim v_{t-m}$ and that $deg(y) = 2$. Since $\Delta = 5$ and $deg(x)=deg(y)=deg(v_1)=2$ we see that in this case $G \cong D_{2m+6}$. Suppose now that $m+4 < t-m$. In this case we can argue that $deg(v_{m+2}) = deg(v_{t-m}) = 5$ and that $v_{m+3} \sim \{v_{m+2}, v_{t-m}\}$ and $v_{t-m-1} \sim \{v_{m+2}, v_{t-m}\}$. Hence $P(k)$ holds for all $k$ such that $t \ge 2k+3$. So if $k$ is the largest integer such that $t \ge 2k+3$, then $t=2k+3$ or $t=2k+4$. So $G \cong S_{2k+4}$ or $G \cong S_{2k+5}$ or $G \cong D_{2k+6}$. This completes the proof.
\end{proof}

\begin{corollary}
\label{cor1}
If $G$ is a locally isometric graph with maximum degree $\Delta \le 5$ and order $n > \Delta$, then $G$ is weakly pancyclic.
\end{corollary}
\begin{proof}
Since the graphs $S_n$ and $D_n$ are weakly pancyclic for all $n \ge 7$, the result follows from Theorem \ref{thm3} and the discussions prior to Theorem \ref{thm3}.
\end{proof}

\section{Cycles in Locally Isometric Graphs with $\Delta \leq 6$}

In this section, we show that all connected, locally isometric graphs with maximum degree $\Delta = 6$ that contain no true twins of degree 6 are either fully cycle extendable or isomorphic to $K_{2,4}+K_1$. Using this result, we then proceed to show that all locally isometric graphs with maximum degree $\Delta = 6$ are weakly pancyclic.

\begin{theorem}
\label{thm1}
If $G$ is a connected, locally isometric graph with maximum degree $\Delta=6$ such that $G$ contains no true twins of degree 6, then $G$ is either fully cycle extendable or isomorphic to $K_{2, 4} + K_1$.
\end{theorem}
\begin{proof}
Let $G$ be a connected, locally isometric graph with maximum degree $\Delta=6$ such that $G$ contains no true twins of degree 6, and suppose  that $C=v_0v_1\ldots v_{t-1}v_0$ is a non-extendable cycle in $G$ for some $t<n$ (indices taken $\bmod~{t}$). Since $t<n$ and $G$ is connected, there must be a vertex of $C$ that is an attachment vertex. We may assume that $v_0$ is an attachment vertex of maximum degree. Let $x\in V(G)-V(C)$ be an off-cycle neighbour of $v_0$.

By Lemma \ref{lemma:nonextendable}(1), we have $x\nsim \left\{{v_{t-1},v_1}\right\}$. Since $diam(\langle N(v_0)\rangle)\leq2$, and $\Delta=6$, we have $4\leq deg(v_0)\leq6$.

\noindent\textbf{Case 1} Suppose $deg(v_0)\leq 5$. If $N(v_0)=\left\{{x,v_1,v_i,v_{t-1}}\right\}$, then by Lemma \ref{lemma:nonextendable}(1), and by $diam(\langle N(v_0)\rangle)\leq2$, $v_i$ must be universal in $\langle N(v_0)\rangle$. Now, since $v_i$ has an off-cycle neighbour $x$, $deg(v_i)\leq deg(v_0)$ by our choice of $deg(v_0)$. Hence $deg(v_i)=4$ and therefore $i-1=1$ and $i+1=t-1$, so $t=4$. Now, Lemmas \ref{lemma:nonextendable}(1) and \ref{lemma:nonextendable}(2) imply that $\{x,v_1,v_3\}$ is an independent set of vertices. By Lemma \ref{full_neighbour}, the only neighbours of each of $x$, $v_1$ and $v_3$ are $v_0$ and $v_2$. So $deg(x)=deg(v_1) =deg(v_3) =2$. We conclude that $G= \langle V(C)\cup \left\{{x}\right\}\rangle$, which contradicts the fact that $\Delta=6$. Hence, $deg(v_0)=5$.

\noindent\textbf{Subcase 1.1} Suppose $\langle N(v_k)\rangle$ contains no universal vertex for each $v_k$ of degree $5$ adjacent to an off-cycle neighbour. Note that this implies that any such $v_k$ cannot have a true twin in $G$. It then follows from the local isometry of $G$ and Lemma \ref{lemma:nonextendable}(1) that cycle vertices have at most one off cycle neighbour. Let $N(v_0)=\left\{{x,v_1,v_i,v_j,v_{t-1}}\right\}$ where $1<i<j<t-1$. Since $\langle N(v_i) \rangle$ is connected, one of $v_i$ or $v_j$ is adjacent with $x$. We may assume $v_i \sim x$. Suppose $v_i \sim v_1$.  If $i+1=j$, then by Lemma \ref{lemma:nonextendable}(1), $x \nsim v_j$. Since $d_{\langle N(v_0)\rangle}(x,v_{t-1})\leq 2$, it follows that $v_i \sim v_{t-1}$. So $v_i$ is a universal vertex in $\langle N(v_0)\rangle$, a contradiction. Hence, $i+1\ne j$. Since $deg(v_0)=5$, and by Lemmas \ref{lemma:nonextendable}(1) and \ref{lemma:nonextendable}(2), it follows that $v_{i+1}\nsim \left\{{x,v_0,v_1}\right\}\subseteq N(v_i)$. By Lemma \ref{true_twin}, $\langle N(v_i) \rangle$ has a universal vertex, contrary to our assumption. We conclude that $v_i \sim x$ implies $v_i \nsim v_1$. Since $d_{\langle N(v_0)\rangle}(x,v_1)\leq 2$ it follows that $v_j\sim \left\{{x,v_{1}}\right\}$. Now  Lemmas \ref{lemma:nonextendable}(1), \ref{lemma:nonextendable}(2) and \ref{lemma:nonextendable}(3) imply that $v_{j+1}\nsim \left\{{x,v_{j-1},v_{1}}\right\}\subseteq N(v_j)$. As for $v_i$ we can argue that $\langle N(v_j) \rangle$ has a universal vertex. Since $deg(v_j) =5$, this contradicts the assumption of this subcase. Hence $d_{\langle N(v_0)\rangle}(x,v_{1})>2$, which contradicts the local isometry of $G$.

\noindent\textbf{Subcase 1.2} Suppose now that there is an attachment vertex $v_k$ such that $\langle N(v_k)\rangle$ contains a universal vertex. We may assume, without loss of generality, that $\langle N(v_0)\rangle$ contains a universal vertex.

\noindent\textbf{Subcase 1.2.1} Suppose that $v_0$ has exactly one off-cycle neighbour $x$. Define the sets $F_1=\left\{{x,v_1}\right\}$, and $F_i=F_{i-1}\cup \left\{{v_{t-i+2},v_i}\right\}$ for $2\leq i\leq \lfloor \frac{t}{2}\rfloor$, where subscripts are expressed modulo $t$. Also, let $N(v_0)=\left\{{x,v_1,v_i,v_j,v_{t-1}}\right\}$, where $1<i<j<t-1$. By Lemma \ref{lemma:nonextendable}(1) and the assumption that $\langle N(v_0)\rangle$ contains a universal vertex, either $v_i$ or $v_j$ is a universal vertex in $\langle N(v_0)\rangle$, say $v_i$. So, by our choice of $v_0$,  $deg(v_i)=deg(v_0)=5$. Hence, $i-1=1$ and $i+1=j$, and thus $i=2$ and $j=3$. By Lemma \ref{lemma:nonextendable}(1), $x \nsim\{v_1,v_3,v_{t-1}\}$ and by Lemma \ref{lemma:nonextendable}(2) $v_1 \nsim \{v_3, v_{t-1}\}$. Since $deg(v_0)=deg(v_2) =5$, it follows, by Lemma \ref{full_neighbour}, that $N(x) =\{v_0,v_2\}$ and hence that $deg(x)=2$. Similarly $deg(v_1)=2$.  We may now conclude that $N[F_2]\subseteq F_3$. Since $v_3 \sim \{v_0,v_2\}$ and $diam(\langle N(v_3)\rangle)\leq 2$, it follows that $v_{t-1}\sim (N[v_3]-\left\{{v_{t-1}}\right\})$. Similarly  $v_{3}\sim (N[v_{t-1}]-\left\{{v_{3}}\right\})$. Hence, $v_{t-1}$ and $v_3$ are true twins, and thus have degree at most $5$ by our assumption that $G$ has no true twins of degree $6$. We also conclude that $N[v_{t-1}]=N[v_3]=\left\{{v_{t-2},v_{t-1},v_0,v_2,v_3,v_4}\right\}$, and hence that $N[F_3]\subseteq F_4$.

 Repeating this argument yields $N[F_{i-1}]\subseteq F_i$, and thus $N[v_{t-i+2}]=N[v_i]=\left\{{v_{t-i+1},v_{t-i+2},v_{t-i+3},v_{i-1},v_i,v_{i+1}}\right\}$ for each $2\leq i \leq \lfloor \frac{t}{2}\rfloor$. If $t$ is odd, we get that $N[v_{\frac{t -1}{2}+1}]=N[v_{\frac{t-1}{2}+2}]= \left\{{v_{\frac{t-1}{2}},v_{\frac{t-1}{2}+1},v_{\frac{t-1}{2}+2},v_{\frac{t-1}{2}+3}}\right\}$, and $G$ can have no additional vertices. If $t$ is even, then $N(v_{\frac{t}{2}+1})=\{{v_{\frac{t}{2}},v_{\frac{t}{2}+2}}\}$, and the only possible additional vertex of $G$ (not in $V(C)\cup \{x\}$) is an off-cycle false twin of $v_{\frac{t}{2}+1}$. In either case, $G$ contains no vertex of degree $6$, a contradiction.

\noindent\textbf{Subcase 1.2.2} Suppose that $v_0$ has two off-cycle neighbours. Let $N(v_0)=\left\{{x_0,y_0,v_1,v_i,v_{t-1}}\right\}$, where $1<i<t-1$. By Lemma \ref{lemma:nonextendable}(1), $v_i$ must be a universal vertex in $\langle N(v_0)\rangle$. Since $v_i$ is an attachment vertex, and by our choice of $v_0$, it follows that $deg(v_i)=5$. Hence, $i-1=1$ and $i+1=t-1$. Thus, $i=2$ and $t=4$. By Lemmas \ref{lemma:nonextendable}(1) and \ref{lemma:nonextendable}(2), $v_1 \nsim v_3$ and $\{v_1,v_3\} \nsim \{x_0, y_0\}$. So, by Lemma \ref{full_neighbour},  $deg(v_1)=deg(v_3)=2$. We conclude that the only vertices in $V(C)\cup\left\{{x_0,y_0}\right\}$ which may have neighbours not in this set are $x_0$ and $y_0$. As $\Delta=6$, $G$ must have additional vertices. Since $diam(\langle N(x_0)\rangle)\leq 2$, it follows that $y_0\sim (N[x_0]-\left\{{y_0}\right\})$. Similarly $diam(\langle N(y_0)\rangle)\leq 2$ implies $x_0\sim (N[y_0]-\left\{{x_0}\right\})$. Hence, $x_0$ and $y_0$ are true twins, and thus have degree at most $5$ by our assumption that $G$ has no true twins of degree $6$. If $n=7$, then there is a vertex $x_1$ such that $N(x_1)=\left\{{x_0,y_0}\right\}$. Moreover, $G\cong S_7$, and $\Delta < 6$, which is not possible. Hence, $n>7$, and there are vertices $x_1$ and $y_1$ such that $\left\{{x_1,y_1}\right\}\sim\left\{{x_0,y_0}\right\}$. We may now conclude that $N[x_0]=N[y_0]=\left\{{v_0,v_2,x_0,y_0,x_1,y_1}\right\}$. If $n=8$, then $G\cong D_8$ and it follows again that $\Delta<6$ which is not possible, so $n>8$. By continuing this argument we see that $G$ is isomorphic to either $S_n$ or $D_n$. This contradicts the fact that $\Delta=6$. Hence Case 1 does not occur.

\noindent \textbf{Case 2}  Suppose that $deg(v_0) = 6$. Since $G$ has no true twins of degree $6$, it follows from Lemma \ref{lemma:nonextendable}(1) and since $diam\langle N(v_0)\rangle\leq2$, that $v_0$ can have at most two off-cycle neighbours.

\noindent \textbf{Subcase 2.1} Suppose $v_0$ has exactly two off-cycle neighbours, $x$ and $y$. Let $N(v_0) = \{x, y, v_1, v_{t-1}, v_i, v_j\}$ where $1 < i < j < t-1$. By Lemma \ref{lemma:nonextendable}(1) we have $\{x, y\} \nsim \{v_{t-1}, v_1\}$. Since $d_{\langle N(v_0)\rangle}(v_{t-1}, x) \leq 2$, it follows that either $v_j \sim \{v_{t-1}, x\}$ or $v_i \sim \{v_{t-1}, x\}$.

\noindent \textbf{Subcase 2.1.1} Suppose that $v_j \sim \{v_{t-1}, x\}$. We further consider the possible ways in which $d_{\langle N(v_0)\rangle}(v_{t-1}, y) \leq 2$.

First, suppose that $v_j \sim y$. Now if $j+1 \neq t-1$, by Lemmas \ref{lemma:nonextendable}(1) and \ref{lemma:nonextendable}(3), and the fact that $\Delta = 6$, we see that $v_{j+1} \nsim \{x, y, v_{j-1}, v_0\} \subseteq \langle N(v_j)\rangle$. By Lemma \ref{true_twin}, $G$ has a true twin of degree $6$, contrary to the hypothesis. Thus, $j+1 = t-1$.
Next, suppose $j-1 \neq i$. By Lemmas \ref{lemma:nonextendable}(1) and \ref{lemma:nonextendable}(2), and the fact that $\Delta = 6$, it follows that $v_{j-1} \nsim \{x, y, v_0, v_{j+1}\} \subseteq \langle N(v_j)\rangle$. So either $deg(v_j)=5$ and $v_{j-1}$ is isolated in $\langle N(v_j) \rangle$ or $deg(v_j) = 6$ and, by Lemma \ref{true_twin}, $v_j$ has a true twin. By our hypothesis, neither case is possible. So $j-1 = i$. Finally, by Lemmas \ref{lemma:nonextendable}(1) and \ref{lemma:nonextendable}(2), and the fact that $\Delta = 6$, we get $v_{t-1} \nsim \{x, y, v_1, v_i\} \subseteq \langle N(v_0)\rangle$, which, by Lemma \ref{true_twin}, is not possible.

Thus $v_j\nsim y$. Since $d_{\langle N(v_0)\rangle}(v_{t-1}, y) \leq 2$, it follows that $v_i \sim \{y, v_{t-1}\}$. By Lemma \ref{lemma:nonextendable}(2), $i+1 \neq j$. Suppose $j+1 \neq t-1$. Then by Lemmas \ref{lemma:nonextendable}(1), \ref{lemma:nonextendable}(2) and \ref{lemma:nonextendable}(3), and the fact that $\Delta = 6$, $v_{j-1} \nsim \{x, v_{t-1}, v_{j+1}, v_0\} \subseteq \langle N(v_j)\rangle$. So either $deg(v_j)=5$ and $v_{j-1}$ is isolated in $\langle N(v_j) \rangle$ or $deg(v_j)=6$ and, by Lemma \ref{true_twin}, $v_j$ has a true twin. Either case contradicts our hypothesis.  Hence $j+1 = t-1$. Using Lemmas \ref{lemma:nonextendable}(1), \ref{lemma:nonextendable}(3) and \ref{lemma:nonextendable}(4), and the fact that $\Delta = 6$, $v_{i+1} \nsim \{y, v_0, v_{i-1}, v_{t-1}\} \subseteq \langle N(v_i)\rangle$. As before we can use Lemma \ref{true_twin} to see that this is not possible.

Thus Subcase 2.1.1 implies that if $v_j \sim x$, then $v_j \nsim v_{t-1}$. A similar argument shows that if $v_j\sim y$, then $v_j\nsim v_{t-1}$, and  if $v_i\sim x$ or $v_i \sim y$, then $v_i\nsim v_1$.

\noindent \textbf{Subcase 2.1.2} Suppose $v_i \sim \{x, v_{t-1}\}$. By the comment following Subcase 2.1.1,  $v_i \nsim v_1$. Since the distance between $v_1$ and both $x$ and $y$ in $\langle N(v_0)\rangle$ is at most 2, $v_j \sim \{x, y, v_1\}$. So, by the above comment, $v_j \nsim v_{t-1}$. Since $d_{\langle N(v_0)\rangle}(v_{t-1}, y) \leq 2$, $v_i \sim y$. Using Lemmas \ref{lemma:nonextendable}(1), \ref{lemma:nonextendable}(2) and \ref{lemma:nonextendable}(3), and the fact that $\Delta = 6$, we conclude that $v_{j+1} \nsim \{x,y,v_0,v_1,v_{j-1}\}$. So $v_{j+1}$ is isolated in $\langle N(v_j)\rangle$, contrary to the local isometry of $G$.

\noindent\textbf{Subcase 2.2} Suppose $v_0$ has exactly one off-cycle neighbour $x$. By Subcase 2.1 we may assume that no vertex of $C$ has two off-cycle neighbours. Let $N(v_0) = \left\{{x,v_1,v_i,v_j,v_k,v_{t-1}}\right\}$ where $1<i<j<k<t-1$. By Lemma \ref{lemma:nonextendable}(1) we have $x\nsim \left\{{v_{t-1},v_1}\right\}$.
We now establish a useful claim that we will use in the remainder of our proof.

\noindent{\bf Claim} $v_k\sim x$ implies $v_k \nsim v_1$ and $v_i \sim x$ implies $v_i \nsim v_{t-1}$. \\
{\em Proof of Claim.} We prove the first of these two statements since the second can be proved in a similar manner. Suppose $v_k \sim \{x, v_1\}$. Now, if $v_{t-1}\nsim v_k$, by Lemmas \ref{lemma:nonextendable}(1), \ref{lemma:nonextendable}(2) and \ref{lemma:nonextendable}(3), and $\Delta=6$, we have $v_{k+1} \nsim \left\{{x,v_1,v_{k-1},v_0}\right\}\subseteq \langle N(v_k)\rangle$. So either $deg(v_k)=5$ and $v_{k+1}$ is isolated in $\langle N(v_k) \rangle$ or $deg(v_k) = 6$ and, by Lemma \ref{true_twin}, $v_k$ has a true twin. Either case contradicts our hypothesis.  Hence, $v_{t-1}\sim v_k$. Thus Lemma \ref{lem5}(1) implies that $k+1=t-1$. If $k-1\ne j$, then by $\Delta=6$ and Lemmas \ref{lemma:nonextendable}(1), \ref{lemma:nonextendable}(2) and \ref{lemma:nonextendable}(3), we have $v_{k-1}\nsim \left\{{x,v_{k+1},v_0}\right\}$, and $v_{k+1}\nsim v_1$. Since $diam(\langle N(v_k)\rangle)\leq2$, and by Lemma \ref{true_twin} and our hypothesis, there is some $v_q \in N(v_k)-\{x,v_0,v_1,v_{k-1}, v_{k+1}\}$ such that $v_q\sim \left\{{x,v_{k-1},v_{k+1}}\right\}$.  Lemma \ref{lemma:nonextendable} and the fact that $\Delta =6$ now imply  that $v_{q-1}\nsim\left\{{x,v_{k},v_{q+1},v_{k+1}}\right\}\subseteq \langle N(v_q)\rangle$. So either $deg(v_q)=5$ and $v_{q-1}$ is isolated in $\langle N(v_q) \rangle$, or $deg(v_q)=6$ and, by Lemma \ref{true_twin}, $v_q$ has a true twin. Either case contradicts our hypothesis.  Thus $k-1=j$.

 By Lemmas \ref{lemma:nonextendable}(1) and \ref{lemma:nonextendable}(2) we have the following non-adjacencies in $\langle N(v_0) \rangle$, $x\nsim \left\{{v_{k-1},v_{t-1},v_1}\right\}$ and $v_{t-1} \nsim \{x,v_1, v_{k-1}\}$. By our hypothesis and Lemma \ref{true_twin}, $v_i \sim \{x, v_{t-1}\}$. If $j \ne i+1$, then Lemmas \ref{lemma:nonextendable}(1), \ref{lemma:nonextendable}(3) and \ref{lemma:nonextendable}(4) imply that $v_{i+1} \nsim \{x, v_0, v_{i-1}, v_{k-1}\}$.  So $i+1=j$.  If $i-1 \ne 1$, then by Lemmas \ref{lemma:nonextendable}(1), \ref{lemma:nonextendable}(2) and \ref{lemma:nonextendable}(3), we have the following non-adjacencies in $\langle N(v_i) \rangle$: $v_{i-1} \nsim \{x, v_0, v_{i+1}, v_{t-1}\}$. As before either $deg(v_i)=5$ and $v_{i-1}$ is isolated in $\langle N(v_i) \rangle$ or $deg(v_i)=6$ and, by Lemma \ref{true_twin} $v_i$ has a true twin. Both cases contradict the hypothesis of the theorem.  So $i-1=1$ and $t=6$.
By Lemmas \ref{lemma:nonextendable}(1) and \ref{lemma:nonextendable}(2) $S_1 =\{x,v_1, v_3,v_5\}$ is an independent set. By the hypothesis $v_0$ and $v_i(=v_2)$ are not true twins. So $v_2 \nsim v_k(=v_4)$, i.e., $S_2 =\{v_2, v_4\}$ is an independent set. We have shown that $S_1 \sim S_2$. Observe that $v_0 \sim (S_1 \cup S_2)$. So $\langle V(C)\cup\left\{{x}\right\}\rangle \cong K_{2,4}+K_1$.  Recall that we have ruled out the existence of an attachment vertex with two off-cycle neighbours. So $N(v_0)$, $N(v_2)$, and $N(v_4)$ are contained in $V(C) \cup \{x\}$. Since $S_1$ is an independent set, it follows, from Lemma \ref{full_neighbour}, that $N(x)$, $N(v_1)$, $N(v_3)$, and $N(v_5)$ are contained in $\{v_0, v_2, v_4\}$. Thus we have $G\cong K_{2, 4} + K_1$, contrary to the hypothesis. This completes the proof of our Claim. $\Box$

\medskip

 \indent By Lemma \ref{lem5}(2) and our hypothesis, $v_j$ is not adjacent with both $v_1$ and $v_{t-1}$. We may assume $v_j \nsim v_{t-1}$. By our Claim $v_i$ is not adjacent with both $x$ and $v_{t-1}$. Since $diam(\langle N(v_0)\rangle) \le 2$, it follows that $v_k \sim \{x, v_{t-1}\}$. If $k+1\ne t-1$, then, by $\Delta=6$, and Lemmas \ref{lemma:nonextendable}(1), \ref{lemma:nonextendable}(2) and \ref{lemma:nonextendable}(3), we have $x \nsim \{v_{k-1}, v_{k+1}, v_{t-1}\}$,  $v_{k+1}\nsim \left\{{x,v_{k-1},v_0}\right\}$, and $v_{k-1}\nsim \{x, v_{t-1}\}$. Since $diam(\langle N(v_k)\rangle)\leq 2$, there is some $w \in N(v_k)-\{x,v_0,v_{k-1}, v_{k+1}, v_{t-1}\}$ such that $w\sim \left\{{x,v_{k-1},v_{k+1}}\right\}$. By Lemma \ref{lemma:nonextendable}(1), $w$ is not an off-cycle vertex, i.e. $w = v_q$ for some $q$. If $q-1\ne k+1$, then Lemmas \ref{lemma:nonextendable}(1), \ref{lemma:nonextendable}(2), \ref{lemma:nonextendable}(3) and the fact that $\Delta=6$ imply $v_{q-1}\nsim\left\{{x,v_{q+1},v_{k-1},v_{k}}\right\}\subseteq \langle N(v_q)\rangle$. So either $deg(v_q)=5$ and $v_{q-1}$ is isolated in $\langle N(v_q) \rangle$ or $deg(v_q)=6$ and $v_q$ has, by Lemma \ref{true_twin}, a true twin. Either case contradicts the hypothesis. Hence  $q-1=k+1$. Thus $\Delta=6$ and Lemmas \ref{lemma:nonextendable}(1), \ref{lemma:nonextendable}(2) and \ref{lemma:nonextendable}(3) imply that $v_{k+1}\nsim\left\{{x,v_{k-1},v_{t-1},v_0}\right\}\subseteq \langle N(v_k)\rangle$. As before we obtain a contradiction to the hypothesis. We conclude that $k+1=t-1$.

 If $j=k-1$, then, by Lemma \ref{lemma:nonextendable}(1), we have $v_j\nsim x$. By our Claim $v_1 \nsim v_k$. Since $d_{\langle N(v_0)\rangle}(v_{1},x)\leq 2$,  it follows that $v_i\sim \{x, v_i\}$. Thus, from Lemmas \ref{lemma:nonextendable}(1) and \ref{lemma:nonextendable}(2), and our Claim, it follows that $v_{t-1}\nsim\left\{{x,v_{j},v_{1},v_{i}}\right\}\subseteq N(v_0)$. By Lemma \ref{true_twin} this contradicts our hypothesis. Hence $j \ne k-1$.

If $x \sim v_i$, then, by our Claim, $v_i \nsim v_{t-1}$. Thus in $\langle N(v_0) \rangle$ we have the following non-adjacencies: $v_{t-1} \nsim \{x, v_1, v_i, v_j\}$. By Lemma \ref{true_twin}, this contradicts our hypothesis. Thus $x \nsim v_i$. Thus in $\langle N(v_0) \rangle$, we have the following non-adjacencies: $x \nsim \{v_1,v_i, v_{t-1}\}$ and $v_1 \nsim \{x, v_k, v_{t-1}\}$. By Lemma \ref{true_twin} and our hypothesis $x \sim v_j$ and $v_1 \sim \{v_i, v_j\}$. By Lemmas  \ref{lemma:nonextendable}(1),  \ref{lemma:nonextendable}(2), and  \ref{lemma:nonextendable}(3) and the fact that $\Delta =6$, $v_{j+1} \nsim \{x,v_0,v_1,v_{j-1}\}$. So either $deg(v_j)=5$ and $v_{j-1}$ is isolated in $\langle N(v_j) \rangle$ or $deg(v_j)=6$ and, by Lemma \ref{true_twin}, $v_j$ necessarily has a true twin. Either case contradicts our hypothesis.

 This completes the proof.
\end{proof}

We now state a useful corollary that follows from the observation that in the proof of Theorem \ref{thm1}, the only cases that do not rely on $G$ having no true twins of degree $6$ on $C$ are Subcase 1.2.2 and the Claim in Subcase 2.2 where it is shown that $G \cong K_{2,4}+K_1$. Observe that in the latter case $G$ has order $7$.

\begin{corollary}
\label{cor2}
Let $G$ be a connected, locally isometric graph of order at least $8$ and with maximum degree $\Delta =6$ that is not fully cycle extendable. Then $G$ has a vertex of degree $2$.
\end{corollary}

\begin{proof}
Let $C=v_0v_1\ldots v_{t-1}v_0$ be a non-extendable cycle in $G$ for some $t<n$. Since $G$ has order at least $8$, the situation in the Claim of Subcase 2.2 mentioned above cannot occur. If follows from Theorem \ref{thm1} that $G$ has some vertex of degree $6$ with a true twin.

\noindent\textbf{Case 1} Suppose no attachment vertex of degree $6$ on $C$ has a true twin.  Since, in the proof of Theorem \ref{thm1}, the only case that did not result in such a vertex occurring on $C$ was Subcase 1.2.2, $\langle V(C)\rangle$ must have the same structure in this case as it did in Subcase 1.2.2, i.e., we will assume $v_0$ has two off-cycle neighbours $x_0$ and $y_0$, that $deg(v_0)=5$ and that $v_0$ has a true twin. Then $t=4$ and $v_2$ is the true twin of $v_0$. By Lemmas \ref{lemma:nonextendable}(1) and \ref{lemma:nonextendable}(2) we see that $\{x_0,y_0\} \nsim v_1$ and $v_1 \nsim v_3$. So, by Lemma \ref{full_neighbour}, $deg(v_1)=2$. So the result holds in this case.

\noindent\textbf{Case 2} Suppose $v_0$ is an attachment vertex of degree $6$ with a true twin $v$. By Lemma \ref{lemma:nonextendable}(1), $v\in V(C)$. Moreover, by Lemma \ref{lemma:nonextendable}(1), and the local isometry of $G$, $v_0$ can have at most three off-cycle neighbours.

\noindent\textbf{Subcase 2.1} Suppose $v_0$ has exactly three off-cycle neighbours, $x,y$ and $z$. Then $N(v_0)=\left\{{x,y,z,v_1,v_i,v_{t-1}}\right\}$ where $1<i<t-1$. It follows from Lemma \ref{lemma:nonextendable}(1) that $v=v_i$. Since $\Delta=6$, it follows that $i-1=1$ and $i+1=t-1$. Hence, $t=4$.  By Lemmas \ref{lemma:nonextendable}(1) and \ref{lemma:nonextendable}(2), $v_1 \nsim \{x,y,z, v_3\}$. Thus, by Lemma \ref{full_neighbour}, $deg(v_1)=2$. So the result holds in this case.

\noindent\textbf{Subcase 2.2} Suppose $v_0$ has exactly two off-cycle neighbours, $x$ and $y$. Then $N(v_0)=\left\{{x,y,v_1,v_i,v_j,v_{t-1}}\right\}$ where $1<i<j<t-1$. It follows from Lemma \ref{lemma:nonextendable}(1) that $v=v_i$ or $v=v_j$. We may assume, without loss of generality, that $v=v_i$. Since $\Delta=6$, we have $i-1=1$ and $i+1=j$. By Lemmas \ref{lemma:nonextendable}(1) and \ref{lemma:nonextendable}(2), $v_1 \nsim\{x,y,v_3, v_{t-1}\}$. So by Lemma \ref{full_neighbour}, $deg(v_1)=2$. So the result holds in this case.

\noindent\textbf{Subcase 2.3} Suppose $v_0$ has exactly one off-cycle neighbour $x$. Then $N(v_0)=\left\{{x,v_1,v_i,v_j,v_k,v_{t-1}}\right\}$ where $1<i<j<k<t-1$. It follows from Lemma \ref{lemma:nonextendable}(1) that $v=v_i$, $v=v_j$ or $v=v_k$. First, suppose that $v=v_j$ is a true twin of $v_0$. Then $\Delta=6$ implies that $j-1=i$ and $j+1=k$. By Lemma \ref{lemma:nonextendable}(1) $x \nsim \{v_1,v_i,v_k,v_{t-1}\}$. Since $deg(v_0)=deg(v_j)=6$, it follows from Lemma \ref{full_neighbour}  that $deg(x)=2$, and the result holds.

If $v \ne v_j$, we may assume, without loss of generality, that $v=v_i$. Then $\Delta=6$ implies that $i-1=1$ and $i+1=j$. Now, if either $deg(x)=2$, or $deg(v_1)=2$, we are done. Suppose this is not the case. It follows, since $\Delta=6$ and from the local isometry of $G$ that any additional neighbours of either $x$ or $v_1$ must be in $N[v_0]=N[v_i]$. Then, by Lemmas \ref{lemma:nonextendable}(1) and \ref{lemma:nonextendable}(2), the only possible additional neighbour for either $x$ or $v_1$ is $v_k$. Since both $v_1$ and $x$ have degree at least $3$,  $N(v_k)=\left\{{x,v_0,v_1,v_2,v_{k-1},v_{k+1}}\right\}$. Recall, by our assumption, that $G$ has order at least $8$. Also, note that if $t=6$, it follows from Lemmas \ref{lemma:nonextendable}(1) and \ref{lemma:nonextendable}(2), the local isometry of $G$, and the fact that $\Delta=6$, that $n=7$. Hence, $t\geq7$. So either $k-1 \ne 3$ or $k+1 \ne t-1$. By symmetry, we may assume, without loss of generality, that $k-1\ne 3$. By Lemmas \ref{lemma:nonextendable}(1), \ref{lemma:nonextendable}(2) and \ref{lemma:nonextendable}(3), and the fact that $\Delta =6$,  $v_{k-1} \nsim \{x, v_0, v_1, v_i, v_{k+1}\}$. So  $v_{k-1}$ is isolated in $\langle N(v_k) \rangle$, a contradiction. We conclude that $G$ must have a vertex of degree $2$.
\end{proof}

We now establish some lemmas which will be useful in proving that locally isometric graphs with maximum degree at most $6$ are weakly pancyclic.

\begin{lemma}
\label{basecase}
Let $G$ be a locally isometric graph of order $n=7$ with $\Delta = 6$. Then $G$ is weakly pancyclic.
\end{lemma}
\begin{proof}We prove this result by showing that every graph of order $7$ and with maximum degree $\Delta = 6$ is fully cycle extendable or weakly pancyclic. Suppose $G$ is not fully cycle extendable and let $C=v_0v_1 \ldots v_{t-1}v_0$ be a non-extendable cycle. Since $n=7$ every vertex of degree $6$ is a universal vertex of $G$. Also since adjacent vertices of $C$ cannot have a common off-cycle neighbour, all vertices of degree 6 are on $C$. We may assume that $deg(v_0) = 6$. If $v_0$ has three off-cycle neighbours $x,y$ and $z$ say, then $t=4$ and $\{x,y,z\} \sim \{v_0, v_2\}$ but $\{x, y, z\} \nsim \{v_1, v_3\}$ and $v_0 \sim v_2$. Depending on whether $\langle \{x,y,z\} \rangle$ contains no, one, or at least two edges, $c(G)$ is $4$, $5$ or $6$, respectively. In all cases it is readily seen that $G$ is weakly pancyclic.

If $v_0$ has two off-cycle neighbours, $x$ and $y$, then $C$ has length $5$. Since $C$ is not extendable $x$ is adjacent with exactly one of $v_2$ and $v_3$. We may assume $x \sim v_2$. Since $G$ is locally isometric, $v_2 \sim v_4$. So, by Lemma \ref{lemma:nonextendable}(2), $y \nsim v_3$. Thus $y \sim v_2$. Depending on whether $x \nsim y$, or $x \sim y$,  $c(G)= 5$ or $6$, respectively. In either case it is easily seen that $G$ is weakly pancyclic.

Suppose $v_0$ has exactly one off-cycle neighbour. Since $v_0$ is adjacent with every vertex on $C$ it is easily seen that $G$ is weakly pancyclic.
\end{proof}

\begin{lemma}
\label{circumferencebound}
If $G$ is a locally isometric graph  that contains a vertex $u$ of degree $2$, then $G-u$ is locally isometric and $c(G) \le c(G-u) +1$.
\end{lemma}
\begin{proof}
Let $v,w$ be the neighbours of $u$ in $G$.  Since $G$ is locally isometric it necessarily follows that $\langle N(u) \rangle$ is connected and hence that $vw \in E(G)$. Moreover, the only neighbourhoods to which $u$ belongs are $N(v)$ and $N(w)$. Since $u$ has degree 1 in both $\langle N(v) \rangle$ and $\langle N(w) \rangle$, it follows that $u$ is not on a shortest path between any pair of vertices in the neighbourhood of any vertex of $G-u$. Hence $G-u$ is locally isometric. If there is a longest cycle $C$ that contains $u$, then $C$ necessarily contains both $v$ and $w$. Hence $G$ contains a cycle of length $c(G) -1$. If no longest cycle of $G$ contains $u$, then $c(G) = c(G-u)$. Hence $c(G) \le c(G-u) +1$.
\end{proof}

\begin{theorem}
Every locally isometric graph with maximum degree $\Delta \le 6$ and order $n \ge \Delta +1$ is  weakly pancyclic.
\end{theorem}
\begin{proof} In Section 3 we showed that if $G$ is locally isometric with maximum degree $\Delta \le 5$, then $G$ is weakly pancyclic. It remains to be shown that all locally isometric graphs with $\Delta = 6$ are weakly pancyclic. We do this by induction on the order.
By Lemma \ref{basecase}, the results holds for $\Delta =6$ and $n=7$. Assume now that $G$ is a locally isometric graph of order $n >7$ with $\Delta =6$ and assume that every locally isometric graph of order $n-1$ and with $\Delta = 6$ is weakly pancyclic. By Corollary \ref{cor2}, $G$ is either fully cycle extendable or $G$ contains a vertex of degree $2$. If $G$ is fully cycle extendable, $G$ is weakly pancyclic. Suppose now that $G$ contains a non-extendable cycle. Hence, $G$ contains a vertex $v$ of degree $2$. So, by Lemma \ref{circumferencebound}, $G-v$ is locally isometric and $\Delta(G-v) \le 6$. If $\Delta(G-v) \le 5$, then $G-v$ is weakly pancyclic by Corollary \ref{cor1}. If $\Delta(G-v) =6$, then it follows from the induction hypothesis that $G-v$ is weakly pancyclic. So, by Lemma \ref{circumferencebound}, $G$ is weakly pancyclic.
\end{proof}

\medskip

In Section 3 we showed that all locally isometric graphs with maximum degree 5 that are not fully cycle extendable are precisely those graphs in ${\mathcal{H}}(m,r,5)$ where $r=1$ or $r=2$.  In this section we showed that if $G$ is a locally isometric graph with maximum degree 6 that is not fully cycle extendable, then $G$ has a pair of true twins of degree 6 or $G \cong K_{2,4}+K_1$. The graphs in the family ${\mathcal{H}}(m,r,6)$ (of $r$-shuttered highrise graphs with maximum degree $6$) described in Section 2, are locally isometric graphs that are not fully cycle extendable but these are not the only such graphs. The graph shown in Fig. \ref{exception}, for example, is another such  graph. It appears unlikely that  a simple structural characterization of these graphs will be found.

\begin{figure}
\begin{center}
\includegraphics*[width=3in]{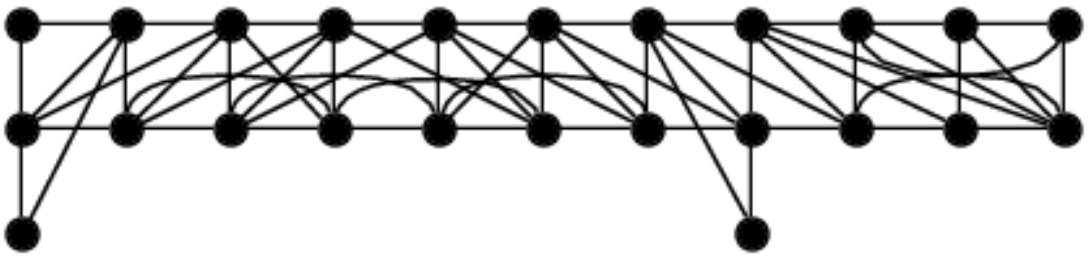}
\caption{A locally isometric graph with $\Delta=6$ that is not fully cycle extendable}
\label{exception}
\end{center}
\end{figure}

\section{The Hamilton Cycle Problem for Locally $k$-Diameter Bounded Graphs}

Recall that a graph $G$ is locally $k$-diameter bounded if $diam(\langle N(v) \rangle) \le k$ for all vertices $v$ of $G$. In this section we establish NP-completeness results for the Hamilton cycle problem in locally $k$-diameter bounded graphs, for $k=2,3$ with  maximum degree at most $8$ or $7$, respectively.

It was remarked in \cite{GJ} that most graph theory problems can be solved in polynomial time for graphs with sufficiently small maximum degree. Nevertheless it was shown in \cite{Pl} that the Hamilton cycle problem for (connected) cubic bipartite planar graphs is NP-complete. This result was used in \cite{GOPS} to show that the Hamilton cycle problem is NP-complete for locally connected graphs with maximum degree $\Delta \le 7$.

We use transformations similar to those used in \cite{GOPS}  to establish the results of this section. Let $H(k, \Delta)$ denote the Hamilton cycle problem in locally $k$-diameter bounded graphs with maximum degree at most $\Delta$. We will use the following terminology in our proof. If $P= v_1v_2 \ldots v_k$ is a path, then the \emph{reversal} of $P$ is the path $P^{\leftarrow} = v_kv_{k-1} \ldots v_1$. A subgraph $H$ of a graph $G$ is a {\em spanning subgraph} if $V(H)=V(G)$.

\begin{theorem}
The problems $H(3,7)$ and $H(2, 8)$ are NP-complete.
\end{theorem}
\begin{proof}
Since the Hamilton cycle problem is in NP, so is $H(k, \Delta)$ for  fixed $k \ge 1$ and $\Delta \ge k+1$.

Let $G$ be a planar cubic bipartite graph. Suppose $U =\{ u_1, u_2, \ldots, u_p \}$ and $V=\{v_1, v_2, \ldots, v_p\}$ are the partite sets of $G$. We now describe a polynomial transformation of $G$ to an instance $\tau_1(G) = G_1$ of $H(3,7)$ and another instance $\tau_2(G)=G_2$ of $H(2,8)$. Since $G$ is a cubic bipartite graph, its edge set can be partitioned into three $1$-factors ${\cal{F}}_1, {\cal{F}}_2$ and ${\cal{F}}_3$. (It is well-known that $1$-factors in regular bipartite graphs can be found in polynomial time, see \cite{bm}.)

To construct $G_1$ and $G_2$ we begin by replacing each vertex $u_i$ and $v_i$, $1 \le i \le p$, of $G$ with a $K_3$. In particular every vertex $u_i \in U$  is replaced with a $K_3$ whose vertices are labeled $u_i^1,u_i^2$ and $u_i^3$ and every vertex $v_i \in V$  is replaced by a $K_3$ whose vertices are labelled $v_i^1, v_i^2$ and $v_i^3$. To complete the construction, suppose a vertex $u_i \sim \{ v_j, v_k, v_l \}$ in $G$ where $u_iv_j \in {\mathcal{F}}_1$, $u_iv_k \in {\mathcal{F}}_2$ and $u_iv_l \in {\mathcal{F}}_3$. For $G_1$ join $u_i^1$ by an edge to each of $v_j^1, v_k^1,v_l^1,v_j^2,v_l^2$ and join $u_i^2$ by an edge to each of $v_j^2, v_k^2,v_l^2,v_k^1,v_l^1$, see Fig. \ref{transformation}.  For $G_2$ join both $u_i^1$ and $u_i^2$ to each of $v_j^1,v_j^2, v_k^1,v_k^2,v_l^1,v_l^2$. So $G_1$ and $G_2$ can be constructed in polynomial time and $\Delta(G_1)=7$ and $\Delta(G_2) =8$.

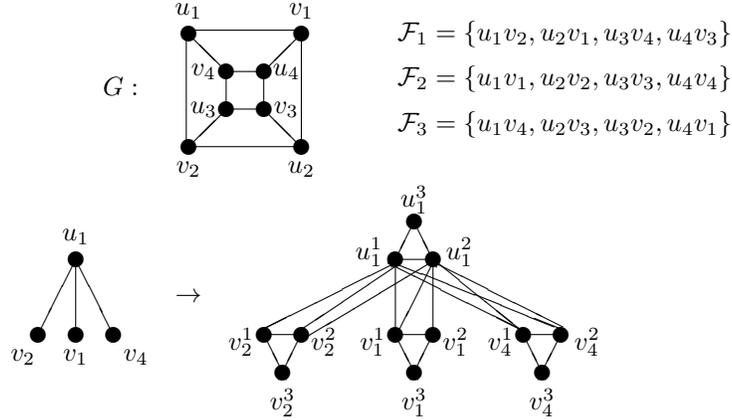
\begin{figure}[htb]
\begin{center}
\begin{picture}(-15,25)

\multiput(-20,20)(15,0){2}{\circle*2}
\multiput(-15,15)(5,0){2}{\circle*2}
\multiput(-15,10)(5,0){2}{\circle*2}
\multiput(-20,5)(15,0){2}{\circle*2}

\put(-20,20.5){\line(1,0){15}}
\put(-20,5){\line(1,0){15}}
\put(-15,15){\line(1,0){5}}
\put(-15,10){\line(1,0){5}}
\put(-20.3,5){\line(0,1){15}}
\put(-5,5){\line(0,1){15}}
\put(-15,10){\line(0,1){5}}
\put(-10,10){\line(0,1){5}}
\put(-20,5){\line(1,1){5}}
\put(-5,5){\line(-1,1){5}}
\put(-20,20){\line(1,-1){5}}
\put(-5,20){\line(-1,-1){5}}

\lab(-29,13){$G:$}

\lab(-20,23){$u_1$}
\lab(-5,23){$v_1$}
\lab(-5,2){$u_2$}
\lab(-20,2){$v_2$}
\lab(-18,15){$v_4$}
\lab(-7,15){$u_4$}
\lab(-7,10){$v_3$}
\lab(-18,10){$u_3$}

\lab(30,20){${\mathcal{F}}_1=\{u_1v_2, u_2v_1, u_3v_4, u_4v_3\}$}
\lab(30,14){${\mathcal{F}}_2=\{u_1v_1, u_2v_2, u_3v_3, u_4v_4\}$}
\lab(30,8){${\mathcal{F}}_3=\{u_1v_4, u_2v_3, u_3v_2, u_4v_1\}$}
\put(-35,-10){\circle*2}

\put(-40,-20){\circle*2}
\put(-35,-20){\circle*2}
\put(-30,-20){\circle*2}

\put(-35,-10){\line(-1,-2){5}}
\put(-35,-10){\line(0,-2){10}}
\put(-35,-10){\line(1,-2){5}}

\lab(-35, -7){$u_1$}
\lab(-35,-23){$v_1$}
\lab(-42,-23){$v_2$}
\lab(-27,-23){$v_4$}


\lab(-20, -15){$\rightarrow$}


\put(10,-5){\circle*2}
\lab(10,-2){$u_1^3$}

\multiput(7.5,-10)(5,0){2}{\circle*2}
\lab(4,-9){$u_1^1$}
\lab(16,-9){$u_1^2$}
\put(10,-5){\line(1,-2){3}}
\put(10,-5){\line(-1,-2){3}}
\put(7.5,-10){\line(1,0){5}}


\multiput(7.5,-20)(5,0){2}{\circle*2}
\lab(4.5,-21){$v_1^1$}
\lab(15.5,-21){$v_1^2$}
\put(10,-25){\circle*2}
\lab(10, -29){$v_1^3$}
\put(7.5,-20){\line(1,0){5}}
\put(7.5,-20){\line(1,-2){3}}
\put(12.5,-20){\line(-1,-2){3}}



\multiput(-10,-20)(5,0){2}{\circle*2}
\lab(-13,-21){$v_2^1$}
\lab(-2,-21){$v_2^2$}
\put(-7.5,-25){\circle*2}
\lab(-7.5, -29){$v_2^3$}

\put(-10,-20){\line(1,0){5}}
\put(-10,-20){\line(1,-2){3}}
\put(-5,-20){\line(-1,-2){3}}



\multiput(24.5,-20)(5,0){2}{\circle*2}
\lab(33,-21){$v_4^2$}
\lab(21.5,-21){$v_4^1$}
\put(27,-25){\circle*2}
\lab(27, -29){$v_4^3$}

\put(24.5,-20){\line(1,0){5}}
\put(24.5,-20){\line(1,-2){3}}
\put(29.5,-20){\line(-1,-2){3}}



\put(7.5,-10.5){\line(-2,-1){18}}
\put(7.5,-10){\line(0,-2){10}}
\put(6.5,-10.5){\line(2,-1){18}}
\put(7.5,-11){\line(-3,-2){13}}
\put(7.5,-10.5){\line(5,-2){22}}
\put(12.5,-10.3){\line(-5,-3){17}}
\put(12.5,-10){\line(0,-2){10}}

\put(12.5,-10.5){\line(-1,-2){5}}
\put(12.5,-10){\line(5,-4){12}}
\put(12.8,-10.5){\line(2,-1){17}}

\end{picture}
\end{center}
\vskip 3 cm \caption{Transforming part of $G$ to part of $G_1$ in the proof of Theorem 5.1}
\label{transformation}
\end{figure}

Let $x \in V(G_1)$. If $x=u_i^3$ or $x=v_i^3$ for some $1 \le i \le p$, then $\langle N(x) \rangle \cong K_2$ and if $x $ is $u_i^1, u_i^2,v_i^1$ or $v_i^2$, then  the graph obtained from $K_{1,5}$ by subdividing exactly one of its edges is a spanning subgraph of $\langle N(x) \rangle $. Hence $G_1$ is locally $3$-diameter bounded. Now suppose $x \in V(G_2)$. If $x=u_i^3$ or $x=v_i^3$ for some $1 \le i \le p$, then $\langle N(x) \rangle \cong K_2$ and if $x $ is $u_i^1, u_i^2,v_i^1$ or $v_i^2$, then $\langle N(x) \rangle$ contains $ K_{1,7}$ as spanning subgraph.  Hence $G_2$ is locally $2$-diameter bounded, i.e. locally isometric.

We show next that $G$ has a Hamiltonian cycle if and only if $G_1$ and $G_2$ have a Hamiltonian cycle. Suppose $G$ has a Hamiltonian cycle $C$. We may assume, without loss of generality, that the vertices of $G$ have been labeled in such a way that
\[C=u_1v_1u_2v_2 \ldots u_pv_pu_1.\]
Then
\[u_1^2u_1^3u_1^1v_1^1v_1^3v_1^2u_2^2u_2^3u_2^1  \ldots u_p^2u_p^3u_p^1v_p^1v_p^3v_p^2u_1^2\]
is a Hamiltonian cycle of both $G_1$ and $G_2$.

Suppose now that $C^1$ is a Hamiltonian cycle of $G_1$ (and $C^2$ is a Hamiltonian cycle of $G_2$).  Since $deg_{G_1}(u_i^3) = deg_{G_1}(v_i^3) =deg_{G_2}(u_i^3)=deg_{G_2}(u_i^3)=deg_{G_2}(v_i^3)=2$, the paths $P_i=u_i^1u_i^3u_i^2$ and $Q_i=v_i^1v_i^3v_i^2$ or their reversals $P_i^{\leftarrow}$ and $Q_i^{\leftarrow}$ must appear in $C^1$ (and $C^2$, respectively). Moreover, the union of the vertices in these paths is $V(G_1)$ (and $V(G_2)$, respectively).  Starting with $u_1^3=u_{i_1}^3$ let $u_{i_1}^3,v_{j_1}^3,u_{i_2}^3,v_{j_2}^3, \ldots , u_{i_p}^3,v_{j_p}^3$ be the order in which the $u_i^3$s and $v_j^3$s appear on $C^1$. Then $u_{i_1}v_{j_1}u_{i_2}v_{j_2} \ldots u_{i_p}u_{j_p}u_{i_1}$ is a Hamiltonian cycle of $G$. An analogous argument for $C^2$ and $G_2$ shows that $G$ has a Hamiltonian cycle if $G_2$ does.

\end{proof}

\section{Concluding Remarks}

Motivated by Ryj\'{a}\v{c}ek's conjecture that every locally connected graph is weakly pancyclic, we investigated global cycle properties of locally $k$-diameter bounded graphs with small maximum degree. We showed that all locally isometric graphs with $\Delta \leq 5$ are weakly pancyclic. Indeed a complete structural characterization of these graphs that are not fully cycle extendable is given. For $\Delta = 6$ we showed that every locally isometric graph of order at least 8 and without a pair of true twins of degree $6$ is fully cycle extendable (and hence weakly pancyclic). Infinite classes of locally isometric graphs with $\Delta=6$ that are not fully cycle extendable are described.  We showed that the Hamilton cycle problem is NP-complete for locally $k$-diameter bounded graphs with maximum degree $\Delta$ for $(k, \Delta)=(3,7)$ and $(k,\Delta)=(2,8)$. The following questions have not yet been answered.

\begin{enumerate}
\item Can the Hamilton cycle problem be solved efficiently for locally $3$- and $4$-diameter bounded graphs with $\Delta =5$?
\item Can the Hamilton cycle problem be solved efficiently for locally $3$-diameter bounded graphs with $\Delta =6$?
\item Is the Hamilton cycle problem for locally $4$- and $5$-diameter bounded graphs with $\Delta =6$  NP-complete?
\end{enumerate}

If the first of these questions can be answered in the affirmative,  then the Hamilton cycle problem for locally connected graphs with $\Delta = 5$ will have been solved.

\section{Acknowledgements} 

\noindent The first and second authors gratefully acknowlege the support of an NSERC USRA Canada Award and the third author gratefully acknowledges the support of an NSERC Discovery Grant Canada 198281-2011.

\end{document}